\documentclass[12pt, reqno]{amsart}

\usepackage{amscd,amsfonts,amsmath,amssymb}
\usepackage[all]{xy}

\usepackage{amssymb}

\newtheorem{theorem}{Theorem}[section]

\newtheorem{lemma}[theorem]{Lemma}
\newtheorem{remark}[theorem]{Remark}

\newtheorem{application}[theorem]{Application}

\numberwithin{equation}{section}

\newcommand{\At}{\text{At}}
\newcommand{\ad}{\text{ad}}
\newcommand{\TM}{\text{TM}}
\newcommand{\GL}{\text{GL}}

\begin{document}
\baselineskip=15.5pt

\title[homogeneous principal bundles]{A criterion for 
homogeneous principal bundles}

\author[I. Biswas]{Indranil Biswas}

\address{School of Mathematics, Tata Institute of Fundamental
Research, Homi Bhabha Road, Bombay 400005, India}

\email{indranil@math.tifr.res.in}

\author[G. Trautmann]{G\"unther Trautmann}

\address{FB Mathematik, Universit\"at Kaiserslautern,
Postfach 3049, D-67653 Kaiserslautern, Germany}

\email{trm@mathematik.uni-kl.de}



\date{}

\begin{abstract}
We consider principal bundles over $G/P$, where $P$ is a parabolic
subgroup of a semisimple and simply connected linear algebraic
group $G$ defined over $\mathbb C$. We prove that a holomorphic
principal $H$--bundle
$E_H\,\longrightarrow\, G/P$, where $H$ is a complex reductive group, is 
homogeneous if the adjoint vector bundle ${\rm ad}(E_H)$ is
homogeneous. Fix a faithful $H$--module $V$. We also show that $E_H$
is homogeneous if the vector bundle $E_H\times^H V$ associated
to it for the $H$--module $V$ is homogeneous.
\end{abstract}

\maketitle

\section{Introduction}\label{sec1}

Let $G$ be a semisimple and simply connected linear algebraic
group defined over $\mathbb C$ and $P\,\subset\, G$ a parabolic
subgroup. So the quotient $G/P$ is a rational complete homogeneous
variety. Let $H$ be any complex algebraic group.
A holomorphic principal $H$--bundle $E_H\, \longrightarrow\,
G/P$ is called homogeneous if the
left--translation action of $G$ on $G/P$ lifts to an action
of $G$ on $E_H$ that commutes with the right action of $H$.
The category of holomorphic homogeneous principal $H$--bundles
coincides with the category of algebraic
homogeneous principal $H$--bundles (see Lemma \ref{lem0}).

Now assume that $H$ is reductive. Fix a finite dimensional
faithful $H$--module $V$. The following is the main result proved
here (see Theorem \ref{thm1}).

\begin{theorem}\label{thm0}
A principal $H$--bundle $E_H\, \longrightarrow\,
G/P$ is homogeneous if and only if the 
associated vector bundle $E_H\times^H V$ is homogeneous.
\end{theorem}

The method of proof of Theorem \ref{thm0} yields the following
(see Lemma \ref{lem1}):

\begin{lemma}
A principal $H$--bundle $E_H\, \longrightarrow\,
G/P$ is homogeneous if and only if its adjoint vector bundle
${\rm ad}(E_H)$ is homogeneous.
\end{lemma}

\begin{application}\rm
In \cite{Sa2}, Sato proved that any infinitely extendable vector
bundle on a nested sequence of homogeneous spaces is homogeneous
(see \cite[p. 171, Main Theorem I]{Sa2} and \cite[p. 171, Main Theorem 
II]{Sa2} for the details). In view of Theorem \ref{thm1}, we
conclude that the results of \cite{Sa2} extend to principal bundles.
In particular, any principal $H$--bundle on an infinite Grassmannian
is homogeneous (see also \cite{Sa1} and \cite{DP}).
Similarly, the results of Penkov and Tikhomirov (see \cite{PT1}, 
\cite{PT2}) extend to principal bundles.
\end{application}

\section{Homogeneous principal bundles and homogeneous vector
bundles}\label{sec2}

Let $G$ be a semisimple and simply connected linear algebraic
group defined over $\mathbb C$, and let $P\, \subset\, G$
be a proper parabolic subgroup. So
\begin{equation}\label{e1}
M\,:=\, G/P
\end{equation}
is a rational complete homogeneous variety. In the following, all
morphisms, as well as all bundles on $M$,
are supposed to be in the holomorphic category; they eventually
may also be algebraic.

The left translation action of $G$ on itself defines an
action of $G$ on the quotient space $M$. 
For any $g\, \in\, G$, let
\begin{equation}\label{ef}
f_g\, :\, M\, \longrightarrow\, M
\end{equation}
be the holomorphic automorphism given by the action of $g$.

Let $H$ be a linear algebraic group defined over $\mathbb C$.
A principal $H$--bundle $E_H$ on $M$ is
called \textit{homogeneous} if the action of $G$ on $M$
lifts to a holomorphic action of $G$ on $E_H$ that commutes
with the right action of $H$. Equivalently, $E_H$ is the extension
of structure group of the principal $P$--bundle $G\,\longrightarrow
\, G/P$ by a homomorphism $P\, \longrightarrow\, H$.

A vector bundle $F$ on $M$ is called \textit{homogeneous}
if the action of $G$ on $M$ lifts to an action
of $G$ on the total space of $F$ which is linear
on the fibers. Thus a vector bundle of rank $n$ over $M$
is homogeneous if and only if the corresponding
principal $\text{GL}(n,{\mathbb C})$--bundle is homogeneous.

The following lemma shows that homogeneous principal
$H$--bundles are algebraically homogeneous.

\begin{lemma}\label{lem0}
Let $E_H\, \longrightarrow\, M$ be a homogeneous principal
$H$--bundle. Then the action of $G$ on $M$
lifts to an algebraic action of $G$ on
$E_H$ satisfying the condition that it commutes
with the right action of $H$.
\end{lemma}

\begin{proof}
Since $E_H$ is homogeneous, for each $g\, \in\, G$, the pulled back
principal $H$--bundle $f^*_g E_H$ is holomorphically isomorphic to 
$E_H$, where
$f_g$ is constructed in \eqref{ef}. Therefore, $f^*_g E_H$
is algebraically isomorphic to $E_H$ for all $g\, \in\, G$.
Now from Proposition 3.1 of \cite{Bi} we conclude that the action
of $G$ on $M$ lifts to an algebraic action of $G$ on $E_H$ satisfying
the condition that it commutes with the right action of $H$.
\end{proof}

Let $H$ be a connected reductive linear algebraic
group defined over $\mathbb C$. Fix a finite
dimensional complex representation
$$
\rho_0\, :\, H\, \longrightarrow\, \text{GL}(V)
$$
such that $\text{kernel}(\rho_0)$ is a finite group. So
the homomorphism of Lie algebras induced by $\rho$
\begin{equation}\label{e3}
{\mathfrak h}\, :=\, \text{Lie}(H)\, \longrightarrow\,
\text{End}(V)
\end{equation}
is injective.

Let $E_H\, \longrightarrow\, M$ be a principal $H$--bundle.
Let
$$
E_V\, :=\, E_H\times^H V\, \longrightarrow\, M
$$
be the vector bundle associated to $E_H$ for the $H$--module
$V$ in \eqref{e3}.

\begin{theorem}\label{thm1}
Assume that the associated vector bundle $E_V$ is homogeneous.
Then the principal $H$--bundle $E_H$ is homogeneous.
\end{theorem}
\vskip3mm

\section{Proof of Theorem \ref{thm1}}

Let $E_{\text{GL}(V)}\,\longrightarrow\, M$ be the principal
$\text{GL}(V)$--bundle corresponding to $E_V$.
By assumption, we have an action
$G\times E_{\text{GL}(V)}\, \longrightarrow\,
E_{\text{GL}(V)}$
of $G$ on $E_{\text{GL}(V)}$ that commutes with the action of 
$\text{GL}(V)$ making $E_{\text{GL}(V)}$ a homogeneous principal
$\text{GL}(V)$--bundle. Our aim is to prove that $E_H$ is 
homogeneous.

Let $\text{At}(E_{\text{GL}(V)})$ (respectively,
$\text{At}(E_H)$) be the Atiyah bundle over $M$ for the principal
$\text{GL}(V)$--bundle $E_{\text{GL}(V)}$
(respectively, principal $H$--bundle $E_H$). We
recall that $\text{At}(E_{\text{GL}(V)})$
(respectively, $\text{At}(E_H)$) is the holomorphic
vector bundle defined by the sheaf of $\text{GL}(V)$--invariant
vector fields on $E_{\text{GL}(V)}$
(respectively, $H$--invariant vector fields on $E_H$); see \cite{At}.
Using the properties of the Atiyah bundle and the injectivity 
of the homomorphism in \eqref{e3} we have the following diagram
with exact rows
\begin{equation}\label{e7}
\begin{xy}
\xymatrix{0\ar[r] & \ad(E_H)\ar[r]\ar@{_{(}->}[d] &
\At(E_H)\ar[r]\ar@{_{(}->}[d] &
\TM\ar[r]\ar@{=}[d] & 0\\
0 \ar[r] & \ad(E_{\GL(V)})\ar[r] & \At(E_{\GL(V)})\ar[r] & \TM\ar[r] & 0
}
\end{xy}
\end{equation}

Since $H$ is reductive, the homomorphism of $H$--modules
in \eqref{e3} splits. In 
other words, there is a submodule $S$ of the $H$--module
$\text{End}(V)$ such that the natural homomorphism
\begin{equation}\label{e10}
{\mathfrak h}\oplus S \, \longrightarrow\, \text{End}(V)
\end{equation}
is an isomorphism of $H$--modules. Fix such a direct summand
$S$, and let
$$
E_S\, :=\, E_H\times^H S
$$ 
be the vector bundle over $M$ associated to the principal $H$--bundle
$E_H$  for the $H$--module $S$. From \eqref{e10} we obtain an 
isomorphism
$$
\ad(E_H)\oplus E_S\,\cong\, \ad(E_{\GL(V)})
$$
of associated vector bundles. Hence from \eqref{e7} we have the
commutative diagram
\begin{equation}\label{atdia2}
\begin{xy}
\xymatrix{0\ar[r] & \ad(E_H)\oplus E_S\ar[d]_\approx\ar[r]^{\iota\oplus 
{\rm Id}_{E_S}} & \At(E_H)\oplus
E_S\ar@{-->}[d]_\approx \ar[r] & \TM\ar@{=}[d]\ar[r] & 0\\
0\ar[r] & \ad(E_{\GL(V)})\ar[r] & \At(E_{\GL(V)})\ar[r] & \TM \ar[r] & 0
}
\end{xy}
\end{equation}
where $\iota\, :\, \ad(E_H)\, \longrightarrow\, \At(E_H)$
is the inclusion in \eqref{e7}. Let
\begin{equation}\label{sigma}
\sigma\, :\, H^0(M,\, \At(E_{\GL(V)}))\, \longrightarrow\,
H^0(M,\, \At(E_H))
\end{equation}
be the surjective homomorphism induced by the projection 
$\At(E_{\GL(V)})\, \longrightarrow\, \At(E_H)$
constructed from \eqref{atdia2}.

Let $\mathfrak g$ denote the Lie algebra of $G$. Let
$$
\varphi\, :\, {\mathfrak g}\, \longrightarrow\,
H^0(M,\, \At(E_{\GL(V)})) \,~\text{~and~}~\,
\varphi_H\, :\, {\mathfrak g}\, \longrightarrow\,
H^0(M,\, \At(E_H))
$$
be the homomorphisms of Lie algebras given by the actions of
$G$ on $E_{\GL(V)}$ and $E_H$ respectively. Note that
$$\varphi_H\, =\, \sigma\circ\varphi\, ,$$ where $\sigma$ is
constructed in \eqref{sigma}.
We have the following commutative diagram of Lie algebras
\begin{equation}\label{liedia}
\begin{xy}
\xymatrix{
{\mathfrak g}\ar[r]^-\varphi\ar[dr]_\beta & 
H^0(M,\At(E_{\GL(V)})\ar[r]^\sigma\ar[d] & H^0(M,\At(E_H))\ar[dl]^\alpha\\
& H^0(M,TM) & 
}
\end{xy}
\end{equation}
where $\beta$ is the injective Lie algebra 
homomorphism induced by the natural action of $G$ on $G/P$,
and $\varphi$, $\sigma$ are defined
above; the remaining two homomorphisms are obtained from \eqref{e7}.
Consequently,
$\beta ({\mathfrak g})\,\subset\, \alpha(H^0(M,\At(E_H)))$. Defining
$$
\widetilde{H^0(M,\At(E_H))}\, :=\, \alpha^{-1}(\beta({\mathfrak g}))\, ,
$$
{}from \eqref{liedia} and \eqref{e7}
we have a short exact sequence of Lie algebras
\begin{equation}\label{e17}
0\, \longrightarrow\, H^0(M,\, \text{ad}(E_H))\, \longrightarrow
\, \widetilde{H^0(M,\, \text{At}(E_H))}
\,\stackrel{\alpha}{\longrightarrow}
\,{\mathfrak g}\, \longrightarrow\, 0\, .
\end{equation}

Since
${\mathfrak g}$ is semisimple, there is a homomorphism of
Lie algebras
\begin{equation}\label{e18}
\widehat{\alpha}\, :\, {\mathfrak g}\, \longrightarrow\,
\widetilde{H^0(M,\, \text{At}(E_H))}
\end{equation}
such that $\alpha\circ \widehat{\alpha}\,=\, \text{Id}_{\mathfrak
g}$;
see \cite[p. 91, Corollaire 3]{Bo}. Fix such a splitting
$\widehat{\alpha}$.

Let ${\mathcal G}(E_H)$ denote the group of all biholomorphisms of
$E_H$ that commute with the right action of $H$. It is a complex Lie
group, and its Lie algebra coincides with $H^0(M,\, \text{At}(E_H))$
(see \cite{Bi} for a proof).
The group $G$ being simply connected, the homomorphism of
Lie algebras $\widehat{\alpha}$ in \eqref{e18} lifts
to a homomorphism
$$
\rho\, :\, G\, \longrightarrow\, {\mathcal G}(E_H)\, .
$$
Since $\alpha\circ \widehat{\alpha}\,=\, \text{Id}_{\mathfrak
g}$, it follows that the action of $G$ on $E_H$ defined
by $\rho$ makes $E_H$ a homogeneous principal $H$--bundle.
This completes the proof of Theorem \ref{thm1}.

\section{Adjoint bundle criterion}

The above proof of Theorem \ref{thm1} also gives the following lemma.

\begin{lemma}\label{lem1}
Assume that the adjoint vector bundle ${\rm ad}(E_H)$ is
homogeneous. Then $E_H$ is homogeneous.
\end{lemma}

\begin{proof}
Define $Z\, :=\, H/[H\, ,H]$. Any holomorphic principal $Z$--bundle
on $G/P$ is homogeneous because $Z$ is a product of copies
of ${\mathbb C}^*$ and $G$ is simply connected (so any line bundle
on $G/P$ is homogeneous).
The $H$--module ${\mathfrak h}$ decomposes as
$$
{\mathfrak h}\,=\, [{\mathfrak h}\, ,{\mathfrak h}]\oplus
z({\mathfrak h})\, ,
$$
where $z({\mathfrak h})$ is the Lie algebra of $Z$. The
adjoint homomorphism ${\mathfrak h}\,\longrightarrow\,
\text{End}_{\mathbb C}({\mathfrak h})$ is injective. So,
the homomorphism of Lie algebras corresponding to the
homomorphism
$$
H\, \longrightarrow\, \text{GL}({\mathfrak h})\times Z\, =:
\, \widetilde{H}
$$
is injective. Let $E_{\widetilde{H}}$ be the principal
$\widetilde{H}$--bundle on $M$ obtained by extending the
structure group of $E_H$ using the above homomorphism.
Since ${\rm ad}(E_H)$ and all holomorphic
line bundles on $G/P$ are homogeneous, and $Z$ is a product
of copies of ${\mathbb C}^*$, it follows that $E_{\widetilde{H}}$
is homogeneous. After replacing the principal
bundle $E_{\text{GL}(V)}$ by
$E_{\widetilde{H}}$, it is straight--forward to check that the
proof of Theorem \ref{thm1} also gives a proof of the lemma.
\end{proof}

\begin{remark}\label{trivial}
{\rm Let $E_H$ be a principal $H$--bundle on $M$, and let $H\,
\longrightarrow\, {\rm GL}(V)$ be a faithful representation
as in \eqref{thm1}. If the associated vector bundle $E_V$ 
is trivial, then it can be shown that $E_H$ itself is trivial.
To prove this, consider the 
induced morphism $E_H\,\longrightarrow\, E_{{\rm GL}(V)}$. The
principal ${\rm GL}(V)$--bundle $E_{{\rm GL}(V)}$ is trivial
because $E_V$
is trivial. Since $M$ is complete and ${\rm GL}(V)/H$ affine, there 
are no
non-constant maps from $M$ to ${\rm GL}(V)/H$. This implies that 
there are trivializing sections of $E_{{\rm GL}(V)}$ which factor 
through $E_H$.
By this remark, the result of \cite{PT1} on the triviality of 
vector bundles on twisted ind--Grassmannians can be extended to
principal bundles. (See \cite{BCT} for
principal bundles on projective spaces.)}
\end{remark}


\end{document}